\documentclass[11pt]{amsart}
\usepackage{color}
\usepackage{amsmath}
\usepackage{amssymb}
\usepackage{amsfonts}
\usepackage[latin1]{inputenc}

\DeclareMathSymbol{\subsetneqq}{\mathbin}{AMSb}{36}

\theoremstyle{plain}
 \numberwithin{equation}{section}
\newtheorem{theorem}{Theorem}[section]

\newtheorem{lemma}{Lemma}[section]

\input{tcilatex}

\begin{document}
\title[Long time behavior for a semilinear hyperbolic equation]{Long time
behavior for a semilinear hyperbolic equation with asymptotically vanishing
damping term and convex potential}
\author{}
\maketitle

\medskip

\centerline{\scshape Ramzi May }

{\footnotesize \ 
\centerline{University of Carthage Tunisia} \centerline {and} %
\centerline{King Faisal University KSA} \centerline{rmay@kfu.edu.sa} }
\bigskip

\bigskip

\noindent \textbf{Abstract} Recently, A. Cabot and P. Frankel studied the
long time behavior of solutions to the following semilinear hyperbolic
equation:

\begin{equation}
\frac{d^{2}u}{dt^{2}}(t)+\gamma (t)\frac{du}{dt}(t)+Au(t)+f(u(t))=0,~\ t\geq
0,  \tag{E}
\end{equation}%
where $\gamma :\mathbb{R}^{+}\rightarrow \mathbb{R}^{+}$, the damping term,
is a decreasing function, $f$ is\ the gradient of a given convex function
defined on an a real Hilbert space $V,$ and $A:V\rightarrow V^{\prime }$ is
a linear and continuous operator assumed to be symmetric, monotone and
semi-coercive. They proved that if the damping term $\gamma (t)$ behaves
like $\frac{K}{t^{\alpha }}$ as $t\rightarrow +\infty ,$ for some $K>0$ and$%
~\alpha \in ]0,1[,$ then every \underline{bounded} solution $u$ to the
equation (E) (i.e. $u\in L^{\infty }(0,+\infty ;V))$ converges weakly in $V$%
\ as $t\rightarrow +\infty $ toward a solution to the stationary equation $%
Av+f(v)=0$. They left open the question: Does convergence still hold without
assuming the boundedness of the solution? In this paper, we give a positive
answer to this question. Our approach relies on precise estimates on the
decay rates for the energy function along trajectories of (E).

\noindent {\textbf{keywords}: Dissipative hyperbolic equation,
asymptotically small dissipation, asymptotic behavior, energy function,
convex function.}

\noindent {\small \textbf{AMS classification numbers:}} 34G10, 34G20, 35B40,
35L70.

\section{Introduction and statement of results}

Throughout this paper, we follow the same notations as in the paper \cite{CF}%
. Let $H$ be a real Hilbert space with inner product and norm respectively
denoted by $\langle .,.\rangle $ and $\left\vert .\right\vert .$ Let $V$ be
a real Hilbert space such that $V\hookrightarrow H\hookrightarrow V^{\prime
} $ with continuous and dense injections, where $V^{\prime }$ is the dual
space of $V.$ Let $\gamma :\mathbb{R}^{+}\longrightarrow \mathbb{R}^{+}$ be
a decreasing function which belongs to the space $W_{loc}^{1,1}\left( 
\mathbb{R}^{+};\mathbb{R}^{+}\right) .$ Let $A:V\rightarrow V^{\prime }$ be
a linear and continuous operator such that the associated bilinear form $%
a:V\times V\rightarrow \mathbb{R}$ $(u,v)\mapsto \langle Au,v\rangle
_{V^{\prime },V}$ is symmetric, positive and satisfies the following
property:%
\begin{equation}
\exists \lambda \geq 0,\mu >0:\forall v\in V,~a(v,v)+\lambda \left\vert
v\right\vert ^{2}\geq \mu \left\Vert v\right\Vert _{V}^{2}.  \label{coer}
\end{equation}%
Let $f:V\rightarrow V^{\prime }$ be a continuous function deriving from a
convex potential i.e, there exists a $C^{1}$ convex function $F:V\rightarrow 
\mathbb{R}$ such that:%
\begin{equation*}
\forall u,v\in V,~F^{\prime }(u)(v)=\langle f(u),v\rangle _{V^{\prime },V}.
\end{equation*}%
It is clear that the function $\phi :V\rightarrow \mathbb{R}$ defined by:%
\begin{equation*}
\phi (v)=\frac{1}{2}a(v,v)+F(v)
\end{equation*}%
is $C^{1},$ convex and satisfies the following property:%
\begin{equation*}
\forall u,v\in V,\phi ^{\prime }(u)(v)=\langle Au+f(u),v\rangle _{V^{\prime
},V}.
\end{equation*}%
We assume moreover that the function $\phi $ is bounded from below and that
the set 
\begin{equation*}
\arg \min \phi =\left\{ v\in V:\phi (v)=\min \phi \right\}
\end{equation*}%
is not empty. Notice that, since $\phi $ is convex, $\arg \min \phi $
coincides with the set $S=\left\{ v\in V:Av+f(v)=0\right\} $ of critical
points of $\phi $.

In this paper, our purpose is to investigate the asymptotic behavior of the
semilinear hyperbolic equation:

\begin{equation}
\frac{d^{2}u}{dt^{2}}(t)+\gamma (t)\frac{du}{dt}(t)+Au(t)+f(u(t))=0,~\ t\geq
0.  \tag{E}
\end{equation}%
This equation and its ODE version (called the heavy ball with friction) have
been studied by many authors under various conditions on the damping and
potential terms, see for instance, \cite{Al}, \cite{AA}, \cite{CEG}, \cite{C}%
, \cite{CF}, \cite{HJ}, \cite{HJ2}, \cite{Z}, and references there in.

\noindent By a solution of (E) we mean a function $u:\mathbb{R}%
^{+}\longrightarrow H$ which belongs to the class 
\begin{equation*}
W_{loc}^{1,1}(\mathbb{R}^{+},V)\cap W_{loc}^{2,1}(\mathbb{R}^{+},H)
\end{equation*}%
and satisfies the equation (E) for almost every $t\geq 0.$ A solution $u$ to
(E) is said to be bounded if it belongs moreover to the space $L^{\infty
}(0,+\infty ;H).$

\noindent In \cite{CF}, Cabot and Frankel proved the following interesting
convergence result:

\begin{theorem}[A. Cabot and P. Frankel]
\label{th1}Assume that there exist $\alpha \in ]0,1[$ and$~K_{1},K_{2}>0$
such that for every $t\geq 0$, $\frac{K_{1}}{(1+t)^{\alpha }}\leq \gamma
(t)\leq \frac{K_{2}}{(1+t)^{\alpha }}.$ Let $u$ be a \underline{bounded}
solution to (E). Then there exists $u_{\infty }\in \arg \min \phi $ such
that $u(t)$ converges weakly in $V$ to $u_{\infty }$ as $t\rightarrow
+\infty .$
\end{theorem}

An open question left in the paper \cite{CF} was whether the condition $u\in
L^{\infty }(0,+\infty ;H)$ is really necessary in the previous theorem (see
Remark 3.15 in \cite{CF}). In the present paper, we will show, without
assuming the boundedness of the solution, that the weak convergence result
still holds in the case $\alpha \in \lbrack 0,\frac{1}{2}]$ and in the case $%
\alpha \in ]\frac{1}{2},1[$ up to a supplementary assumption on the
derivative of the damping term $\gamma .$ Such assumption is satisfied, for
instance, by functions of the form $\frac{K}{(1+t)^{\alpha }}$ where $K>0.$
Moreover, in each case, we will establish an estimate on the rate of the
decay for the energy function on the trajectories of (E). More precisely, we
will prove the following theorems:

\begin{theorem}
\label{th2}Assume that there exist $\alpha \in \lbrack 0,\frac{1}{2}]$ and$%
~K>0$ such that for every $t\geq 0$, $\gamma (t)\geq \frac{K}{(1+t)^{\alpha }%
}.$Then for every solution $u$ to (E) there exists $u_{\infty }\in \arg \min
\phi $ such that $u(t)$ converges weakly in $V$ to $u_{\infty }$ as $%
t\rightarrow +\infty .$ Moreover, 
\begin{equation*}
\frac{1}{2}\left\vert \frac{du}{dt}(t)\right\vert ^{2}+\phi (u(t))-\min \phi
=\circ (\frac{1}{t})\text{ as }t\rightarrow +\infty .
\end{equation*}
\end{theorem}

\begin{theorem}
\label{th3}Assume that there exist $\alpha \in \lbrack 0,1[$, $K>0$ and $%
t_{0}\geq 0$ such that $\gamma (t)\geq \frac{K}{(1+t)^{\alpha }}$ for every $%
t\geq 0$ and $\gamma ^{\prime }(t)\leq -\alpha \frac{\gamma (t)}{1+t}$ for
almost every $t\geq t_{0}.$ Let $u$ be a solution to (E), then $u(t)$
converges weakly in $V$ as $t\rightarrow +\infty $ toward some $u_{\infty
}\in \arg \min \phi $. Moreover, for every $\bar{\alpha}<\alpha ,$%
\begin{equation*}
\frac{1}{2}\left\vert \frac{du}{dt}(t)\right\vert ^{2}+\phi (u(t))-\min \phi
=\circ (\frac{1}{t^{1+\bar{\alpha}}})\text{ as }t\rightarrow +\infty .
\end{equation*}
\end{theorem}

\section{Proof of Theorem \protect\ref{th2} and Theorem \protect\ref{th3}}

We will first prove some preliminary results under the following general
hypothesis on the damping term $\gamma :$%
\begin{equation}
\exists K>0\text{ and }\alpha \in \lbrack 0,1[:\forall t\geq 0,~\gamma
(t)\geq \frac{K}{(1+t)^{\alpha }}.  \label{hyp}
\end{equation}%
These results will be useful in the proofs of Theorem \ref{th2} and Theorem %
\ref{th3}.

\noindent Let $u$ be a solution to the equation (E). Define the energy
function 
\begin{equation}
\mathcal{E}(t)=\frac{1}{2}\left\vert \frac{du}{dt}(t)\right\vert ^{2}+\phi
(u(t))-\min \phi ,~t\geq 0.  \label{energy1}
\end{equation}%
A simple computation yields%
\begin{equation*}
\frac{d\mathcal{E}}{dt}(t)=-\gamma (t)\left\vert \frac{du}{dt}(t)\right\vert
^{2},~a.e.~t\geq 0.
\end{equation*}%
Thus the function $\mathcal{E}$ is decreasing and converges as $t\rightarrow
+\infty $ to some real number $\mathcal{E}_{\infty }$ which will be
identified later. Moreover%
\begin{equation}
\int_{0}^{+\infty }\gamma (t)\left\vert \frac{du}{dt}(t)\right\vert
^{2}dt<\infty  \label{energy2}
\end{equation}%
and 
\begin{equation}
\forall t\geq 0,~\mathcal{E}(t)-\mathcal{E}_{\infty }=\int_{t}^{+\infty
}\gamma (s)\left\vert \frac{du}{dt}(s)\right\vert ^{2}ds.  \label{energy3}
\end{equation}%
Let $v$ be a fixed point in $\arg \min \phi $ and define the function $p(t)=%
\frac{1}{2}\left\vert u(t)-v\right\vert ^{2},t\geq 0.$ Proceeding as in the
proof of Proposition 3.5 in \cite{CF}, one can easily prove that for almost
every $t$ in $\mathbb{R}^{+}$ we have%
\begin{equation*}
\ddot{p}(t)+\gamma (t)\dot{p}(t)\leq \frac{3}{2}\left\vert \frac{du}{dt}%
(t)\right\vert ^{2}-\mathcal{E}(t).
\end{equation*}%
Multiplying the last inequality by $\lambda _{r}(t)=(1+t)^{r},$ $r\in 
\mathbb{R},$ and integrating by parts over the interval $[0,T],~T>0$, we
easily obtain after simplification%
\begin{eqnarray}
\int_{0}^{T}\lambda _{r}(t)\mathcal{E}(t)dt &\leq &\frac{3}{2}%
\int_{0}^{T}\lambda _{r}(t)\left\vert \frac{du}{dt}(t)\right\vert
^{2}dt-\lambda _{r}(T)\dot{p}(T)+\left[ \lambda _{r}^{\prime }-\left( \gamma
\lambda _{r}\right) \right] (T)p(T)  \notag \\
&&+\int_{0}^{T}\left[ \left( \lambda _{r}\gamma \right) ^{\prime }-\lambda
_{r}^{\prime \prime }\right] (t)p(t)dt+C_{r}  \label{energy4}
\end{eqnarray}%
where $C_{r}=\dot{p}(0)+(\gamma (0)-r)p(0).$

\noindent Since $\gamma $ satisfies (\ref{hyp}), $\lambda _{r}^{\prime
}(T)=\circ \left[ \left( \gamma \lambda _{r}\right) (T)\right] $ as $%
T\rightarrow +\infty .$ Thus, there exits $T_{r}\geq 0$ such 
\begin{equation}
\forall t\geq T_{r},~\lambda _{r}^{\prime }(T)-\left( \gamma \lambda
_{r}\right) (T)\leq -\frac{1}{2}\left( \gamma \lambda _{r}\right) (T).
\label{est1}
\end{equation}%
On the other hand, thanks to Cauchy-Schwarz inequality, we have%
\begin{eqnarray}
\left\vert \dot{p}(T)\right\vert &\leq &\left\vert \frac{du}{dt}%
(t)\right\vert \left\vert u(t)-v\right\vert  \label{est2} \\
&\leq &2\sqrt{\mathcal{E}(T)}\sqrt{p(T)}  \notag
\end{eqnarray}%
Inserting estimates (\ref{est1})-(\ref{est2}) into (\ref{energy4}) and using
hypothesis (\ref{hyp}) and the following elementary inequality%
\begin{equation*}
\forall a>0~\forall ~x,b\in \mathbb{R},~bx-ax^{2}\leq \frac{b^{2}}{4a},
\end{equation*}%
with $x=\sqrt{p(T)},$ we deduce that for every $T\geq T_{r}$ we have%
\begin{eqnarray}
\int_{0}^{T}\lambda _{r}(t)\mathcal{E}(t)dt &\leq &\frac{3}{2}%
\int_{0}^{T}\lambda _{r}(t)\left\vert \frac{du}{dt}(t)\right\vert ^{2}dt+%
\frac{2}{K}\lambda _{r+\alpha }(T)\mathcal{E}(T)  \notag \\
&&+\int_{0}^{T}\left[ \left( \lambda _{r}\gamma \right) ^{\prime
}(t)-\lambda _{r}^{\prime \prime }(t)\right] p(t)dt+C_{r}.  \label{Energy}
\end{eqnarray}%
Let us notice that if $r\leq 0,$ $\left( \lambda _{r}\gamma \right) ^{\prime
}(t)-\lambda _{r}^{\prime \prime }(t)\leq 0$ $a.e.$ on $\mathbb{R}^{+}$
(since the function $\lambda _{r}\gamma $ is decreasing and the function $%
\lambda _{r}$ is convex); then, in the case where $r\leq 0,$ (\ref{Energy})
becomes 
\begin{equation}
\forall T\geq T_{r},~\int_{0}^{T}\lambda _{r}(t)\mathcal{E}(t)dt\leq \frac{3%
}{2}\int_{0}^{T}\lambda _{r}(t)\left\vert \frac{du}{dt}(t)\right\vert ^{2}dt+%
\frac{2}{K}\lambda _{r+\alpha }(T)\mathcal{E}(T)+C_{r}.  \label{E1}
\end{equation}%
Letting $r=-\alpha $ in the last inequality and using (\ref{energy2}) and
the fact that is $\mathcal{E}$ a decreasing function, we get%
\begin{equation*}
\forall T\geq T_{-\alpha },~\int_{0}^{T}\lambda _{-\alpha }(t)\mathcal{E}%
(t)dt\leq \frac{3}{2K}\int_{0}^{+\infty }\gamma (t)\left\vert \frac{du}{dt}%
(t)\right\vert ^{2}dt+\frac{2}{K}\mathcal{E}(0)+C_{-\alpha },
\end{equation*}%
which implies that%
\begin{equation}
\int_{0}^{+\infty }\lambda _{-\alpha }(t)\mathcal{E}(t)dt<\infty .
\label{E2}
\end{equation}%
Recalling that $\alpha <1,$ we then deduce that the limit $\mathcal{E}%
_{\infty }$\ of $\mathcal{E}(t)$ as $t\rightarrow +\infty $ is equal to zero.

\noindent Let us now prove the following crucial lemma:

\begin{lemma}
\label{le1}Let $r\in \mathbb{R}\backslash \{-1\}.$ If $\int_{0}^{+\infty
}\lambda _{r}(t)\mathcal{E}(t)dt<\infty $ then $\mathcal{E}(t)=\circ
(1/t^{1+r})$ as $t\rightarrow +\infty $ and $\int_{0}^{+\infty }\lambda
_{r+1-\alpha }(t)\left\vert \frac{du}{dt}(t)\right\vert ^{2}dt<\infty .$
\end{lemma}

\begin{proof}
Since the energy function $\mathcal{E}$ is decreasing, we have 
\begin{equation}
\mathcal{E}(t)\int_{\frac{t}{2}}^{t}(1+s)^{r}ds\leq \int_{\frac{t}{2}%
}^{+\infty }\lambda _{r}(s)\mathcal{E}(s)ds.  \label{R}
\end{equation}%
A simple computation yields $\int_{\frac{t}{2}}^{t}(1+s)^{r}ds\simeq
M_{r}t^{r+1}$ for $t$ large enough where $M_{r}$ is a nonnegative constant
depending only on $r$. Inserting this last estimate into (\ref{R}), we get $%
\lim_{t\rightarrow +\infty }t^{1+r}\mathcal{E}(t)=0.$ On the other hand, by
using equality (\ref{energy3}), the fact that $\mathcal{E}_{\infty }=0,$ and
Fubini Theorem, we obtain%
\begin{equation*}
\int_{0}^{+\infty }\lambda _{r}(t)\mathcal{E}(t)dt=\frac{1}{1+r}%
\int_{0}^{+\infty }\gamma (s)\left[ (1+s)^{r+1}-1\right] \left\vert \frac{du%
}{dt}(s)\right\vert ^{2}ds,
\end{equation*}%
which clearly implies that $\int_{0}^{+\infty }\lambda _{r+1-\alpha
}(t)\left\vert \frac{du}{dt}(t)\right\vert ^{2}dt<\infty $ since $%
\int_{0}^{+\infty }\gamma (s)\left\vert \frac{du}{dt}(s)\right\vert
^{2}ds<\infty $ and $\gamma (s)\geq \frac{K}{(1+s)^{\alpha }}.$
\end{proof}

Now we are in position to complete the proof of our first main
theorem.\smallskip

\noindent \textbf{Proof of Theorem \ref{th2}:} In view of (\ref{E2}), Lemma %
\ref{le1} implies $\mathcal{E}(t)=\circ (t^{\alpha -1})$ as $t\rightarrow
+\infty $ and $\int_{0}^{+\infty }\lambda _{1-2\alpha }(t)\left\vert \frac{du%
}{dt}(t)\right\vert ^{2}dt<\infty .$ Hence by letting $r=0$ in (\ref{E1}),
we get, for $T$ large enough, 
\begin{equation*}
\int_{0}^{T}\mathcal{E}(t)dt\leq \frac{3}{2}\int_{0}^{T}\left\vert \frac{du}{%
dt}(t)\right\vert ^{2}dt+\circ (T^{2\alpha -1})+C_{0}.
\end{equation*}%
Therefore, by letting $T\rightarrow +\infty $ and using the assumption $%
\alpha \leq \frac{1}{2},$ we get%
\begin{equation*}
\int_{0}^{\infty }\mathcal{E}(t)dt\leq \frac{3}{2}\int_{0}^{\infty }\lambda
_{1-2\alpha }(t)\left\vert \frac{du}{dt}(t)\right\vert ^{2}dt+C_{0},
\end{equation*}%
Hence, by using once again Lemma \ref{le1}, we deduce that $\mathcal{E}%
(t)=\circ (1/t)$ as $t\rightarrow +\infty $ and that $\int_{0}^{+\infty
}\lambda _{1-\alpha }(t)\left\vert \frac{du}{dt}(t)\right\vert ^{2}dt<\infty 
$ which implies, since $\alpha \leq \frac{1}{2},$ that $\int_{0}^{+\infty
}(1+t)^{\alpha }\left\vert \frac{du}{dt}(t)\right\vert ^{2}dt<\infty $.
Therefore we deduce the weak convergence of $u(t)$ in $V$ as $t\rightarrow
+\infty $ from the following lemma which is implicitly proved in \cite{CF}
(see the proofs of Theorem 3.7 and Theorem 3.13) by adapting a classical
arguments originated by F. Alvarez \cite{Al} based on the famous Opial's
lemma \cite{Op}.

\begin{lemma}
\label{le3}Assume (\ref{hyp})$.$ Let $u$ be a solution to (E). If $%
\int_{0}^{\infty }(1+t)^{\alpha }\left\vert \frac{du}{dt}(t)\right\vert
^{2}dt<\infty $ then $u(t)$ converges weakly in $V$ as $t\rightarrow +\infty 
$ to some $u_{\infty }\in \arg \min \phi .$
\end{lemma}

Now we are going to prove our second main theorem. Hence, hereafter, we
assume that the function $\gamma $ satisfies (\ref{hyp}) and the hypothesis
on its derivative given in Theorem \ref{th3}. First we will prove the
following key lemma:

\begin{lemma}
\label{le2}If $\nu <2\alpha -1$ and $\int_{0}^{+\infty }\lambda _{\nu }(t)%
\mathcal{E}(t)dt<+\infty $ then $\int_{0}^{+\infty }\lambda _{\nu +1-\alpha
}(t)\mathcal{E}(t)dt<+\infty .$
\end{lemma}

\noindent \textbf{Proof of Lemma \ref{le2}:} Let $\nu <2\alpha -1$ such that 
$\int_{0}^{+\infty }\lambda _{\nu }(t)\mathcal{E}(t)dt<+\infty .$ According
to Lemma \ref{le1}, we have:%
\begin{equation}
\mathcal{E}(t)=\circ (1/t^{1+\nu })\text{ as }t\rightarrow +\infty
\label{H1}
\end{equation}%
and 
\begin{equation}
\int_{0}^{+\infty }\lambda _{1+\nu -\alpha }(t)\left\vert \frac{du}{dt}%
(t)\right\vert ^{2}dt<\infty .  \label{H2}
\end{equation}%
Let $\rho =1+\nu -\alpha .$ Using the hypothesis on the damping term $\gamma 
$ and the fact that $\rho <\alpha ,$ we find that for almost every $t\geq
t_{0}$ we have 
\begin{eqnarray*}
\left[ \left( \lambda _{\rho }\gamma \right) ^{\prime }-\lambda _{\rho
}^{\prime \prime }\right] (t) &\leq &\left( \rho -\alpha \right) \lambda
_{\rho -1}(t)\gamma (t)-\rho (\rho -1)\lambda _{\rho -2}(t) \\
&\leq &\left( \rho -\alpha \right) K~\lambda _{\rho -\alpha -1}(t)-\rho
(\rho -1)\lambda _{\rho -2}(t) \\
&\simeq &\left( \rho -\alpha \right) K~\lambda _{\rho -\alpha -1}(t)\text{
as }t\rightarrow +\infty .
\end{eqnarray*}%
The last inequality implies that there exists $\tau _{0}\geq \max
(T_{0},t_{0})$ such that for almost every $t\geq \tau _{0}$ we have $\left[
\left( \lambda _{\rho }\gamma \right) ^{\prime }-\lambda _{\rho }^{\prime
\prime }\right] (t)\leq 0.$ Inserting this last inequality into (\ref{Energy}%
) with $r=\rho ,$ we obtain%
\begin{equation}
\int_{0}^{T}\lambda _{\rho }(t)\mathcal{E}(t)dt\leq \frac{3}{2}%
\int_{0}^{T}\lambda _{\rho }(t)\left\vert \frac{du}{dt}(t)\right\vert ^{2}dt+%
\frac{2}{K}\lambda _{1+\nu }(T)\mathcal{E}(T)+A_{\rho }\text{ for a.e. }%
T\geq \tau _{0},  \label{H3}
\end{equation}%
where $A_{\rho }=C_{\rho }+\int_{0}^{\tau _{0}}\left[ \left( \lambda
_{r}\gamma \right) ^{\prime }(t)-\lambda _{r}^{\prime \prime }(t)\right]
p(t)dt.$ Hence, by using estimates (\ref{H1})-(\ref{H2}) and by letting $%
T\rightarrow +\infty $ in (\ref{H3}), we deduce that $\int_{0}^{+\infty
}\lambda _{\rho }(t)\mathcal{E}(t)dt<\infty .\smallskip $

Now we are in position to prove our second main theorem.\smallskip

\noindent \textbf{Proof of Theorem \ref{th3}:} We will proceed as in the
proof of Theorem 1.3 in \cite{JM}. Let $A=\{\nu \in \mathbb{R}%
:\int_{0}^{+\infty }\lambda _{\nu }(t)\mathcal{E}(t)dt<+\infty \}$. From (%
\ref{Energy}), $-\alpha \in A$, thus $A$ is a non empty interval of $\mathbb{%
R}$ which is on the forme $A=]-\infty ,\alpha _{0}[$ or $A=]-\infty ,\alpha
_{0}]$ where $\alpha _{0}=\sup A$. The previous lemma asserts that: if $\nu
<\alpha _{0}$ and $\nu <2\alpha -1$ then $\nu +1-\alpha \leq \alpha _{0}$
which means that $\min (\alpha _{0},2\alpha -1)\leq \alpha _{0}+\alpha -1.$
Now since $\alpha -1<0,$ the last inequality reads as $2\alpha -1\leq \alpha
_{0}+\alpha -1$, thus $\alpha \leq \alpha _{0}$. Therefore, by using the
defintion of $\alpha _{0}$ and Lemma \ref{le1} we infer that for all $\bar{%
\alpha}<\alpha $, $\mathcal{E}(t)=\circ (1/t^{1+\bar{\alpha}})$ as $%
t\rightarrow +\infty $ and $\int_{0}^{+\infty }(1+t)^{1+\bar{\alpha}-\alpha
}\left\vert \frac{du}{dt}(t)\right\vert ^{2}dt<\infty .$ Hence, by taking $%
\bar{\alpha}$ closed enough to $\alpha $ and using the fact that $\alpha <1,$
we deduce that $\int_{0}^{+\infty }(1+t)^{\alpha }\left\vert \frac{du}{dt}%
(t)\right\vert ^{2}dt<\infty $ which completes the proof thanks to Lemma \ref%
{le3}.\smallskip

\noindent \textbf{Acknowledgement:} The author wish to thank Prof. Mohamed
Ali Jendoubi for his comments, remarks and suggestions which were very
useful to improve the results and the representation of the paper.


\begin{thebibliography}{99}
\bibitem{Al} F. Alvarez, On the minimizing properties of a second order
dissipative system in Hilbert spaces. \textit{SIAM J. Cont. Optim. } 38%
\textbf{\ (}4\textbf{)} (2000) 1102-1119.

\bibitem{AA} F. Alvarez and H. Attouch, Convergence and asymptotic
stabilization for some damped hyperbolic equation with non-isolated
equilibria. \textit{ESAIM} 6 (2000) 1-34.

\bibitem{CEG} A. Cabot, H. Engler, and S. Gadat, On the long time behavior
of second order differential equations with asymptotically small
dissipation. \textit{Trans. Amer. Math. Soc. }361 (11) (2009) 5983-6017.

\bibitem{C} A. Cabot, H. Engler, and S. Gadat, Second order differntial
equations with asymptotically small dissipation and piecewise flat
potentials. \textit{Electron. J. Differential Equations} 17 (2009) 33-38.

\bibitem{CF} A. Cabot and P. Frankel, Asymptotics for some semilinear
hyperbolic equations with non-autonomous damping. \textit{J. Differential
Equations} 252 (2012) 294-322.

\bibitem{HJ} A. Haraux and M.A. Jendoubi, Convergence of solutions of
second-order gradient-like systems with analytic nonlinearities. \textit{J.
Differential Equations} 144 (1998) 313-320.

\bibitem{HJ2} A. Haraux and M.A. Jendoubi, Asymptotics for a second order
differential equation with a linear slowly time-decaying damping term. 
\textit{Evolution Equations and Control Theory} 2 (3) (2013) 461-470.

\bibitem{JM} M.A. Jendoubi and R. May , Asymptotics for a second-order
differential equation with non-autonomous damping and an integrable source
term. \textit{Applicable Analysis} (2014) DOI: 10.1080/00036811.2014.

\bibitem{Op} Z. Opial, Weak convergence of the sequence of successive
aproximation for nonexpansive mapping, \textit{Bull. Amer. Math. Soc.} 73
(1967) 591-597.

\bibitem{Z} E. Zuazua, Stability and decay for a class of nonlinear
hyperbolic problems. \textit{Asymptotic Anal.} 1 (1998) 161-185.
\end{thebibliography}
\end{document}